\documentclass[12pt,reqno]{amsart}
\usepackage{a4wide,amsfonts,amsmath,latexsym,amssymb,euscript,eufrak,graphicx,units,mathrsfs}
\usepackage[utf8]{inputenc}
\usepackage{amsmath}
\usepackage{amsfonts}
\usepackage{amssymb}
\usepackage{amsthm}
\usepackage{floatrow}
\usepackage{blindtext}
\usepackage{multicol}
\usepackage[english]{babel}
\usepackage{enumerate}
\usepackage{eufrak}
\usepackage{graphicx}
\usepackage{caption}
\usepackage{subcaption}
\usepackage{float}
\usepackage{epstopdf}
\usepackage{multirow}
\usepackage{mathtools}
\usepackage{multirow}
\usepackage[usenames,dvipsnames]{color}

\numberwithin{equation}{section}

\newcommand{\inner}[2]{\langle #1, #2 \rangle}

\newtheorem{theorem}{Theorem}[section]

\newtheorem{corollary}[theorem]{Corollary}

\theoremstyle{definition}

\newtheorem{example}[theorem]{Example}

\numberwithin{equation}{section}
\begin{document}

\renewcommand{\thefootnote}{\arabic{footnote}}

\begin{center}
{\Large \textbf{A  quantitative CLT on a finite   sum of
Wiener chaoses and applications to   ratios of Gaussian functionals }} \\[0pt]
~\\[0pt]
Khalifa Es-Sebaiy\footnote{%
Department of Mathematics, Faculty of Science, Kuwait University,
Kuwait. Email: \texttt{khalifa.essebaiy@ku.edu.kw}} 
\\[0pt]
\textit{Department of Mathematics, Faculty of Science, Kuwait University }\\[0pt]
~\\[0pt]
\end{center}

{\ \noindent \textbf{Abstract:} In this paper we provide a new explicit bound  on the total variation distance
between a standardized partial sum of random variables belonging to a finite sum of Wiener chaoses and a standard normal random variable. 
 We apply our result to derive an upper bound  for the
Kolmogorov  distance between  a ratio    of multiple stochastic integrals and a
  Gaussian random variable.
\vspace*{0.1in}}

{\ \noindent }\textbf{Mathematics Subject Classifications (2020)}:  60G15; 60F05; 60H07.


\section{\textbf{Introduction}} 
Fix integers $d\geq1, \ N\geq2$, and consider a  random variable $F$ expressed as a   partial sum   of  multiple stochastic integrals:
\begin{align}F=\sum_{p=d}^NI_p(f_{p})\, \mbox{ with }  f_{p}\in {\mathcal{H}}^{\odot
p}, \ p=d,\ldots,N,\label{decomp-F-intro}\end{align} where  ${\mathcal{H}}$ is a  Hilbert space associated to   an isonormal Gaussian process  $W=\{W(h): h~\in~{\mathcal{H}}\}$, and 
${\mathcal{H}}^{\odot p}$ denotes the $p$th symmetric tensor power of ${\mathcal{H}}$.\\
Our main result, the forthcoming Theorem \ref{CLT-finite-sum}, states that if $F$ is a random variable  of the form  \eqref{decomp-F-intro}, then there exists a constant $C>0$ depending only
on $N$  such that
\begin{eqnarray}&&d_{TV}\left( \frac{F}{\sqrt{E\left[F^{2}\right] }}%
, \mathcal{N}(0,1)\right)\nonumber\\&& \leqslant \frac{C}{E\left[F^{2}\right]}
\left(\max_{\underset{1\leq r\leq p-1}{d\vee2\leq p\leq N}}\Vert
f_{p}\underset{r}{{\otimes }}f_{p}\Vert_{{\mathcal{\mathcal{H}}}%
^{\otimes {2p-2r}}}+\delta_{dN}\max_{d\leq p< q\leq N}\sqrt{\langle
f_{p} \otimes f_{p},f_{q} \underset{q - p}{{\otimes
}} f_{q}\rangle _{{\mathcal{\mathcal{H}}}%
^{\otimes 2p}}}\right),\label{Intro-main-ineq.1}
\end{eqnarray}
 where $\delta_{dN}=0$ if $d=N$, and $\delta_{dN}=1$ if $d\neq N$.\\
Note that the estimate \eqref{Intro-main-ineq.1} is an  extension of    \cite[Theorem 1]{EV}:
\begin{theorem}[{\cite[Theorem 1]{EV}}] Let $F=\sum_{k=2}^q I_k\left(f_k\right)$ with $q \geqslant 2$ is a fixed integer and $f_k \in \mathcal{H}^{\odot k}, k=2, \ldots, q$.  Then, there exists a constant $C>0$ depending only on $q$ such that
\begin{eqnarray}
d_{T V}\left(\frac{F}{\sqrt{E\left[F^2\right]}}, \mathcal{N}(0,1)\right) \leqslant \frac{C}{E\left[F^2\right]} \max _{2 \leqslant k \leqslant q}\Vert f_k\Vert_{\mathcal{H}^{\otimes k}} \times \max _{\substack{1 \leqslant s \leqslant k \\ 2 \leqslant k \leqslant q}}\sqrt{ \|f_k \underset{s}{{\otimes
}} f_k \|_{\mathcal{H}^{\otimes 2 k-2 s}}} .
\label{estimate-EV}
\end{eqnarray}
\end{theorem}
More precisely, the estimate \eqref{estimate-EV} can be  immediately obtained from \eqref{Intro-main-ineq.1} by applying  Cauchy-Schwarz inequality to the term  
$\sqrt{\langle
f_{p} \otimes f_{p},f_{q} \underset{q - p}{{\otimes
}} f_{q}\rangle _{{\mathcal{\mathcal{H}}}%
^{\otimes 2p}}}$ and using the fact that $\Vert f_{p} \otimes f_{p}\Vert_{\mathcal{\mathcal{H}}^{\otimes 2p}}= \left\Vert f_p\right\Vert^2_{\mathcal{H}^{\otimes p}}$. 

To motivate our results, let us briefly consider the particular case of \eqref{Intro-main-ineq.1} when $F$ is of the form $F = I_2(f) + I_4(g),\, f\in \mathcal{H}^{\odot 2}, g\in \mathcal{H}^{\odot 4}$, which was proved by \cite[Theorem 8]{DEAV}. This particular case was used  to estimate the speed of convergence in the central limit theorem (CLT)  of the quadratic variation for AR(1) processes driven by second-chaos white noise. Moreover, according to \cite[Remark 12]{DEAV}, it turns out that the   bound in \eqref{Intro-main-ineq.1} is sharper than the bound in \eqref{estimate-EV}. For this reason, when we apply  the estimate \eqref{Intro-main-ineq.1}, the term $\sqrt{\langle
f_{p} \otimes f_{p},f_{q} \underset{q - p}{{\otimes
}} f_{q}\rangle _{{\mathcal{\mathcal{H}}}%
^{\otimes 2p}}}$  cannot merely be bounded above via
Cauchy-Schwarz inequality. Thus,  the improved bound in \eqref{Intro-main-ineq.1} is easily applicable in concrete situations,
 and   gives rise to sharp bounds in a variety of applications.

Now let us discuss an application to   Breuer–Major theorem. Consider a centered stationary Gaussian sequence $\left\{
Z_{k}, k\geq0\right\}$ with covariance $\rho(k):=\mathbb{E}(Z_{0}Z_{k})$ for all $k\geq0$ such that $\rho(0)>0$,  and $|\rho(k)|\leq1$ for large~$k$. We put
$\rho(k)=\rho(-k)$ for $k<0$. Fix integers  $d,m \geq   1$, and define  
\begin{align}F_n=\frac{1}{\sqrt{n}}\sum_{i=0}^{n-1}g(Z_i),\label{decomp2-F-intro}\end{align} where the function $g$ can be expressed in the
sub-basis of even-rank Hermite polynomials as follows:
\[g(x)=\sum_{k=d}^{m} \lambda_{g, 2 k} H_{2 k}\left({x}/{\sqrt{\rho(0)}}\right),\]
for some  $\lambda_{g, 2 k} \in \mathbb{R}, k=d, \ldots, m$. Then,  we can estimate the rate of convergence
in terms of the covariance function $\rho$ for the total variation distance
between ${F_n}/{\sqrt{\mathbb{E}[F_n^2]}}$ 
and the standard normal law $\mathcal{N}(0,1)$, see the forthcoming Corollary \ref{corol-stat}. In this context, if $F_n$ is a  sequence of the form \eqref{decomp2-F-intro} with is a general function, then, under different regularity assumptions on $g$,  several explicit upper bounds for $d_{T V}\left({F_n}/{\sqrt{E\left[F_n^2\right]}}, \mathcal{N}(0,1)\right)$ have been   recently obtained by \cite{KN}, \cite{NNP} and \cite{NZ}.

Finally, we apply the forthcoming Theorem \ref{CLT-finite-sum}, especially the forthcoming Corollary \ref{1st-2nd-chaos}, to derive new explicit bounds on the Kolmogorov   distance between a random variable $Q_T$ of the form \[Q_T= \frac{I_2(g_T)+I_1(f_T)+a_T}{ I_2(h_T)+b_T}, \quad g_T,h_T\in {\mathcal{H}}^{\odot2}, f_T\in {\mathcal{H}}, a_T,b_T\in \mathbb{R},\] and a
standard Gaussian random variable, see Theorem \ref{main-thm} below. This result, when $Q_T$ has the form \[Q_T= \frac{I_2(g_T)+a_T}{ I_2(h_T)+b_T}, \quad g_T,h_T\in {\mathcal{H}}^{\odot2},  a_T,b_T\in \mathbb{R},\] was proved by \cite{EA2023}, and was used (see \cite{AAE} and \cite[Section 4]{EA2023}) to   provide   Berry–Esseen bounds in
 Kolmogorov distance for   maximum
 likelihood drift estimation for linear stochastic   differential equations.
 
 The rest of this paper is structured as follows. In Section 2, we present some basic elements  of Gaussian
 analysis and Malliavin calculus used in this paper. Our main result, Theorem \ref{CLT-finite-sum}, is presented in Section \ref{sec:partial-sum}, as well as an application to    Breuer–Major theorem. Finally, in
 Section  \ref{sect:ratio}, we consider an application  to a ratio of  multiple  stochastic   integrals.

\section{\textbf{Elements of Malliavin calculus on Wiener space}}\label{sec:review}

 In this section,  we briefly recall some facts on Malliavin calculus and
 Stein’s method that are needed in the proofs of this paper. The interested reader can find more details in \cite{NP-book} and \cite{nualart-book}.

Let ${\mathcal{H}}$ be a real separable Hilbert space with inner product $\langle\cdot, \cdot\rangle_{{\mathcal{H}}}$ and norm $\|\cdot\|_{{\mathcal{H}}}$; in order to simplify our
 discussion, we will assume for the rest of the paper that ${\mathcal{H}}=L^2(A)$  with a Polish space $(A, \mathcal{A})$ and a non-atomic $\sigma$-finite measure $\mu$. We  denote for integers  $q \geq 1$  by ${\mathcal{H}}^{\otimes q}$ the $q$th tensor power and by ${\mathcal{H}}^{\odot q}$ the $q$th symmetric tensor power of ${\mathcal{H}}$. In this case  we have that ${\mathcal{H}}^{\otimes q}=$ $L^2\left(A^q, \mathcal{A}^{\otimes q}, \mu^{\otimes q}\right)$ and that ${\mathcal{H}}^{\odot q}$ coincides with the space $L_{\text {sym }}^2\left(A^q\right)$ of $\mu^{\otimes q}$-almost everywhere symmetric functions on $A^q$. Moreover,  for integers $p, q \geq 1, f \in {\mathcal{H}}^{\odot p}, g \in {\mathcal{H}}^{\odot q}$ and $r \in\{1, \ldots, \min \{p, q\}\}$,  the $r$th contraction of $f$ and $g$,  denoted by $f \otimes_r g$, is defined as
$$
\begin{array}{rl}
\left(f \otimes_r g\right)\left(y_1, \ldots, y_{p+q-2 r}\right):= & \int_{A^r}f\left(x_1, \ldots, x_r, y_1, \ldots, y_{p-r}\right) \\
& \quad \times g\left(x_1, \ldots, x_r, y_{p-r+1}, \ldots, y_{p+q-2 r}\right) \mu^{\otimes r}\left(\mathrm{~d}\left(x_1, \ldots, x_r\right)\right).
\end{array}
$$
Moreover, $f \otimes_0 g=f \otimes g$ equals the tensor product of $f$ and $g$ while $f \otimes_q g=\langle f, g\rangle_{{\mathcal{H}}^{\otimes q}}$ if $p=q$. Note that, in general, the contraction $f \otimes_r g$ is not a symmetric element of ${\mathcal{H}}^{\otimes(p+q-2 r)}$.  
\\
 We now let $W=\{W(h): h \in {\mathcal{H}}\}$ denote an isonormal Gaussian process over ${\mathcal{H}}$, which  is a family consisting 
of centred Gaussian random variables defined on a probability space $(\Omega, \mathcal{F}, \mathbb{P})$ and indexed by the elements of ${\mathcal{H}}$,   with covariance  
$$
 \mathbb{E}\left[W(h) W\left(h^{\prime}\right)\right]=\left\langle h, h^{\prime}\right\rangle_{{\mathcal{H}}}, \quad h, h^{\prime} \in {\mathcal{H}} .
$$
Without loss of generality we can assume that the $\sigma$-field $\mathcal{F}$ is generated by $W$. Let  $L^2(\Omega)$ denote  the space of square-integrable functions over $\Omega$. The $q$th Wiener chaos $\mathcal{H}_q$  of order $q$ associated with $W$ is the closed linear subspace of  $L^2(\Omega)$ generated by random variables of the form $H_q(W(h))$, where $H_q$ is the $q$th Hermite polynomial and $h \in {\mathcal{H}}$ satisfies $\|h\|_{{\mathcal{H}}}=1$. For $q = 0$ put $\mathcal{H}_0:=\mathbb{R}$. The $q$-th multiple integral of $h^{\otimes q} \in {\mathcal{H}}^{\odot q}$ is defined by the identity $I_q\left(h^{\otimes q}\right)=H_q(W(h))$ for any $h\in {\mathcal{H}}$ with $\|h\|_{{\mathcal{H}}}=1$. The map $I_q$ provides a linear isometry between ${\mathcal{H}}^{\odot q}$ (equipped with the norm $\sqrt{q!}\|\cdot\|_{{\mathcal{H}}^{\otimes q}}$ ) and $\mathcal{H}_q$ (equipped with $L^2(\Omega)$ norm). By convention, $I_0(c)=c$ for all $c\in\mathbb{R}$. We put $I_q(h):=I_q(\tilde{h})$ for general $h \in {\mathcal{H}}^{\otimes q}$ where $\tilde{h} \in {\mathcal{H}}^{\odot q}$ is the canonical symmetrization of $h$.\\
The main motivation for introducing the multiple integrals comes from the following properties:
\begin{itemize}
\item \textit{Wiener chaos expansion of $F\in L^2\left(\Omega\right)$~\cite[Theorem 2.2.4]{NP-book}}. $L^2\left(\Omega\right)$ can be decomposed into an infinite orthogonal sum of Wiener chaoses ${\mathcal{H}}_q$, $q \geq0$. In other words, any $F\in L^2\left(\Omega\right)$ can be represented as
    $$
F=\sum_{q=0}^{\infty} I_q\left(h_q\right)
$$
with $h_0=\mathbb{E}[F]$ and uniquely determined elements $h_q \in {\mathcal{H}}^{\odot q}, q \geq 1$.
\item \textit{Isometry property of integrals~\cite[Proposition 2.7.5]{NP-book}}. Fix integers $p, q \geq 1$ as well as $f \in {\mathcal{H}}^{\odot p}$ and $g \in {\mathcal{H}}^{\odot q}$.
\begin{align*}
  \mathbb{E} [ I_p(f) I_q(g) ] = \left\{ \begin{array}{ll} p! \inner{f}{g}_{{\mathcal{H}}^{\otimes p}} & \mbox{ if } p = q \\ 0 & \mbox{otherwise.} \end{array} \right.
\end{align*}
\item \textit{Product formula~\cite[Proposition 2.7.10]{NP-book}}. Fix integers $p,q \geq 1$. If $g \in {\mathcal{H}}^{\odot p}$ and $h \in {\mathcal{H}}^{\odot q}$, then
\begin{align*}
  I_p(g) I_q(h) = \sum_{r = 0}^{p \wedge q} r! {p \choose r} {q \choose r} I_{p + q -2r}(g \widetilde{\otimes}_r h).
\end{align*}
\end{itemize}
For a smooth and cylindrical random variable $F=f\left(W\left(\varphi_1\right), \ldots, W\left(\varphi_n\right)\right)$, with $\varphi_j \in {\mathcal{H}}$ and $f \in C_b^{\infty}\left(\mathbb{R}^n\right)$ (meaning that $f$ and its partial derivatives are bounded), the Malliavin derivative of $F$ is defined as the ${\mathcal{H}}$-valued random variable:
$$
D F=\sum_{j=1}^n \frac{\partial f}{\partial x_j}\left(W\left(\varphi_1\right), \ldots, W\left(\varphi_n\right)\right) \varphi_j .
$$
By iteration, we can also define the $k$th derivative $D^k F$, which is an element in the space $L^2\left(\Omega ; {\mathcal{H}}^{\otimes k}\right)$. For any real $p \geq 1$ and any integer $k \geq 1$, the Sobolev space $\mathbb{D}^{k, p}$ is defined as the closure of the space of smooth and cylindrical random variables with respect to the norm $\|\cdot\|_{k, p}$ defined by

$$
\|F\|_{k, p}^p=\mathbb{E}\left(|F|^p\right)+\sum_{j=1}^k \mathbb{E}\left(\left\|D^j F\right\|_{{\mathcal{H}}^{\otimes j}}^p\right) .
$$
 Let $Z$ be a real-valued random variable such that $\mathbb{E}[Z^4] <\infty$. The third and fourth
cumulants of $Z$ are, respectively, defined by
$$
\begin{aligned}
&\kappa_{3}(Z)=\mathbb{E}\left(Z^{3}\right)-3 \mathbb{E}\left(Z^{2}\right) \mathbb{E}(Z)+2 \left[\mathbb{E}(Z)\right]^{3}, \\
&\kappa_{4}(Z)=\mathbb{E}\left(Z^{4}\right)-4 \mathbb{E}(Z) \mathbb{E}\left(Z^{3}\right)-3
\left[\mathbb{E}\left(Z^{2}\right)\right]^{2}+12 \left[\mathbb{E}(Z)\right]^{2} \mathbb{E}\left(Z^{2}\right)-6 \left[\mathbb{E}(Z)\right]^{4}.
\end{aligned}
$$
Moreover, if $\mathbb{E}(Z)=0$, we obtain  
\[\kappa_{3}(Z)=\mathbb{E}\left(Z^{3}\right)\ \mbox{ and }\ \kappa_{4}(Z)=\mathbb{E}\left(Z^{4}\right)-3
\left[\mathbb{E}\left(Z^{2}\right)\right]^{2}.\]
Recall that, given two real-valued random variables $X,Y$, the total variation distance between the law of $X$ and the law of $Y$
is defined as
$$
d_{TV}(X, Y):=\sup_{B}\left|\mathbb{P}\left(X\in B\right)-\mathbb{P}\left(Y\in B\right)\right|,
$$
where  the supremum runs over the class of all Borel subsets $B$ of $\mathbb{R}$, and the
Kolmogorov distance  between the law of $X$ and the law of $Y$  is defined  as
\begin{equation*}
d_{Kol}\left( X,Y\right):= \sup_{z\in \mathbb{R}}\left\vert
\mathbb{P}\left(X\leq z\right)-\mathbb{P}\left( Y\leq
z\right)\right\vert.
\end{equation*}
Note that \begin{equation}d_{Kol}\left( X,Y\right)\leq d_{TV}\left( X,Y\right).\label{dKol:dTV}\end{equation}

The following notations will be used throughout this paper. We use $\mathcal{N} (0, 1)$  to denote a standard
normal random variable, while $\mathcal{N} (\mu, \sigma^2)$ denotes
a normal  variable with mean $\mu$ and variance $\sigma^2$. Also,  $C$ denotes a generic positive constant, which may change
from line to line.

\section{\textbf{A quantitative CLT for partial sums of multiple  stochastic  integrals}}\label{sec:partial-sum}
 The aim of the present section  is to measure the  total variation distance between a standardized partial sum of random variables belonging to a finite sum of Wiener chaoses and a standard normal random variable.

\begin{theorem}\label{CLT-finite-sum} Let $d\geq1, \ N\geq2$ be fixed integers, and let $F=\sum_{p=d}^NI_p(f_{p})$, with $f_{p}\in {\mathcal{H}}^{\odot
p}, \ p=1,\ldots,N$. Then there exists a constant $C>0$ depending only
on $N$  such that
\begin{eqnarray}&&d_{TV}\left( \frac{F}{\sqrt{E\left[F^{2}\right] }}%
, \mathcal{N}(0,1)\right) \nonumber\\&&\leqslant \frac{C}{E\left[F^{2}\right]}
\left(\max_{\underset{1\leq r\leq p-1}{d\vee2\leq p\leq N}}\Vert
f_{p}\underset{r}{{\otimes }}f_{p}\Vert_{{\mathcal{\mathcal{H}}}%
^{\otimes {2p-2r}}}+\delta_{dN}\max_{d\leq p< q\leq N}\sqrt{\langle
f_{p} \otimes f_{p},f_{q} \underset{q - p}{{\otimes
}} f_{q}\rangle _{{\mathcal{\mathcal{H}}}%
^{\otimes 2p}}}\right),\label{main-ineq.1}
\end{eqnarray}
where $\delta_{dN}=0$ if $d=N$, and $\delta_{dN}=1$ if $d\neq N$.
\end{theorem}
\begin{proof}
Fix   integers $d\geq 1$ and $N\geq2$, and let
\[F=\sum_{p=d}^NI_p(f_{p}),\ \mbox{ with }\ f_{p}\in {\mathcal{H}}^{\odot
p}, \ p=1,\ldots,N.\]
 It is shown in \cite[Proposition 2.4]{NP2013} that
\begin{equation}
d_{TV}\left( \frac{F}{\sqrt{E\left[F^{2}\right] }}%
,\mathcal{N}(0,1)\right) \leqslant \frac{2}{E\left[F^{2}\right]}
E\left\vert E\left[F^{2}\right]-\left\langle
DF,-DL^{-1}F\right\rangle _{\mathcal{H}}\right\vert.\label{ineq1}
\end{equation}%
On the other hand, in view of  the fact that
\begin{equation*}
E\left[ E\left( \left( I_{p}(f_{p})\right) ^{2}\right) -\langle
DI_{p}(f_{p}),-DL^{-1}I_{p}(f_{p})\rangle _{\mathcal{H}}\right] =0,
\quad \mbox{and }\ E\left[I_p(f_p)I_q(f_q)\right]=0\ \mbox{ if }
p\neq q,
\end{equation*}%
we obtain
\begin{eqnarray*}
&&E\left\vert E\left[F^{2}\right]-\left\langle
DF,-DL^{-1}F\right\rangle _{\mathcal{H}}\right\vert \\&\leqslant&
\sum_{p=d\vee2}^{N}E\left| E\left( \left( I_{p}(f_{p})\right)
^{2}\right)
-\frac{1}{p}\|DI_{p}(f_{p})\|^2_{\mathcal{H}}\right|+\sum_{d\leqslant
p\neq q\leqslant N}\frac{1}{q}E\left\vert \langle
DI_{p}(f_{p}),DI_{q}(f_{q})\rangle _{\mathcal{H}}\right\vert\\
&\leqslant&
\sum_{p=d\vee2}^{N}\sqrt{Var\left(\frac{1}{p}\Vert DI_{p}(f_{p})\Vert _{%
\mathcal{\mathcal{H}}}^{2}\right) }+\sum_{d\leqslant p\neq
q\leqslant N}\sqrt{E\left[\left(\frac{1}{q} \langle
DI_{p}(f_{p}),DI_{q}(f_{q})\rangle _{\mathcal{H}}\right)^2\right]}.
\end{eqnarray*}%
Moreover, thanks to \cite[Lemma 3.1]{NPP}, we have, for $p\geq2$,
\begin{eqnarray*}
Var\left( \frac{1}{p}\Vert DI_{p}(f_{p})\Vert _{\mathcal{\mathcal{H}}%
}^{2}\right) &=&\frac{1}{p^2}\sum_{r=1}^{p-1}r^{2}r!^{2}\binom{p}{r}^{4}(2p-2r)!\Vert f_{p}\widetilde{\underset{r}{\otimes}}f_{p}\Vert _{{%
\mathcal{\mathcal{H}}}^{\otimes 2p-2r}}^{2}\\
&\leq&\frac{1}{p^2}\sum_{r=1}^{p-1}r^{2}r!^{2}\binom{p}{r}^{4}(2p-2r)!\Vert f_{p} \underset{r}{\otimes}f_{p}\Vert _{{%
\mathcal{\mathcal{H}}}^{\otimes 2p-2r}}^{2},
\end{eqnarray*}%
and for $p,q\geq1$,  
\begin{eqnarray*}
&&E\left[ \left( \frac{1}{q}\langle
DI_{p}(f_{p}),DI_{q}(f_{q})\rangle _{\mathcal{H}}\right) ^{2}\right]
\leqslant p^{2} \sum_{r=1}^{p\wedge
q}(r-1)!^{2}\binom{p-1}{r-1}^{2}\binom{q-1}{r-1} ^{2}(p+q-2r)! \Vert
f_{p}\widetilde{\underset{r}{{\otimes }}}f_{q}\Vert
_{{\mathcal{\mathcal{H}}}^{\otimes p+q-2r}}^{2}.
\end{eqnarray*}
Further, for $1\leq r<p\wedge q$,
\begin{eqnarray*}
\Vert f_{p}\widetilde{\underset{r}{{\otimes }}}f_{q}\Vert
_{{\mathcal{\mathcal{H}}}^{\otimes p+q-2r}}^{2} &\leq& \Vert f_{p}
\underset{r}{{\otimes }}f_{q}\Vert
_{{\mathcal{\mathcal{H}}}^{\otimes p+q-2r}}^{2}\\
&=&\langle
f_{p}\underset{p-r}{{\otimes }}f_{p},f_{q}\underset{q-r}{{\otimes }}f_{q}\rangle _{{\mathcal{\mathcal{H}}}%
^{\otimes 2r}}\\
&\leq&  \Vert f_{p}\underset{p-r}{{\otimes }}f_{p}\Vert
_{{\mathcal{\mathcal{H}}}^{\otimes 2r}} \Vert
f_{q}\underset{q-r}{{\otimes }}f_{q}\Vert _{{\mathcal{\mathcal{H}}}%
^{\otimes 2r}}\\
&\leq& \frac12\left( \Vert f_{p}\underset{p-r}{{\otimes }}f_{p}\Vert
_{{\mathcal{\mathcal{H}}}^{\otimes 2r}}^{2}+\Vert
f_{q}\underset{q-r}{{\otimes }}f_{q}\Vert _{{\mathcal{\mathcal{H}}}%
^{\otimes 2r}}^{2}\right)\\
&=& \frac12\left( \Vert f_{p}\underset{r}{{\otimes }}f_{p}\Vert
_{{\mathcal{\mathcal{H}}}^{\otimes 2p-2r}}^{2}+\Vert
f_{q}\underset{r}{{\otimes }}f_{q}\Vert _{{\mathcal{\mathcal{H}}}%
^{\otimes 2q-2r}}^{2}\right),
\end{eqnarray*}%
where we used in the latter equation the fact that $r\geq1$, $p-r<p$
and $q-r<q$ (see \cite[page 121]{NP-book}).\\
For $r=p\wedge q$,
\begin{eqnarray*}
\Vert f_{p\wedge q}\widetilde{\underset{p\wedge q}{{\otimes
}}}f_{p\vee q}\Vert _{{\mathcal{\mathcal{H}}}^{\otimes p\vee
q-p\wedge q}}^{2} &\leq& \Vert f_{p\wedge q} \underset{p\wedge
q}{{\otimes }}f_{p\vee q}\Vert _{{\mathcal{\mathcal{H}}}^{\otimes
p\vee
q-p\wedge q}}^{2}\nonumber\\
&=&\langle f_{p\wedge q} \otimes f_{p\wedge q},f_{p\vee q}
\underset{p\vee q - p\wedge q}{{\otimes
}} f_{p\vee q}\rangle _{{\mathcal{\mathcal{H}}}%
^{\otimes 2(p\wedge q)}}\nonumber\\
&=&\langle f_{p\wedge q} \otimes f_{p\wedge q},f_{p\vee q}
\underset{|q - p|}{{\otimes
}} f_{p\vee q}\rangle _{{\mathcal{\mathcal{H}}}%
^{\otimes 2(p\wedge q)}}.\label{case r=inf(p,q)}
\end{eqnarray*}
The above observations show that 
\begin{eqnarray}
&&E\left\vert E\left[F^{2}\right]-\left\langle
DF,-DL^{-1}F\right\rangle _{\mathcal{H}}\right\vert \nonumber\\
&\leqslant&
\sum_{p=d\vee2}^{N}\frac{1}{p}\sqrt{\sum_{r=1}^{p-1}r^{2}r!^{2}\binom{p}{r}^{4}(2p-2r)!\Vert f_{p} \underset{r}{\otimes}f_{p}\Vert _{{%
\mathcal{\mathcal{H}}}^{\otimes 2p-2r}}^{2} }\nonumber\\&&+\sum_{d\leqslant
p\neq q\leqslant N}\frac{p}{\sqrt{2}}\sqrt{ \sum_{r=1}^{p\wedge
q-1}(r-1)!^{2}\binom{p-1}{r-1}^{2}\binom{q-1}{r-1} ^{2}(p+q-2r)!
\left( \Vert f_{p}\underset{r}{{\otimes }}f_{p}\Vert
_{{\mathcal{\mathcal{H}}}^{\otimes 2p-2r}}^{2}+\Vert
f_{q}\underset{r}{{\otimes }}f_{q}\Vert _{{\mathcal{\mathcal{H}}}%
^{\otimes 2q-2r}}^{2}\right)}\nonumber\\&&+\sum_{d\leqslant p\neq q\leqslant
N} p (p\wedge q-1)!\binom{p\vee q-1}{p\wedge q-1}\sqrt{|q-p|!
\langle f_{p\wedge q} \otimes f_{p\wedge q},f_{p\vee q} \underset{|q
- p|}{{\otimes
}} f_{p\vee q}\rangle _{{\mathcal{\mathcal{H}}}%
^{\otimes 2(p\wedge q)}}}.\label{ineq2}
\end{eqnarray}
Therefore, in view of \eqref{ineq1} and \eqref{ineq2},  the estimate \eqref{main-ineq.1} is obtained.
\end{proof}

\begin{corollary}\label{1st-2nd-chaos}
 Let $F:= I_1(f_{1})+I_2(f_{2})$, with $f_{1}\in {\mathcal{H}}, f_{2}\in {\mathcal{H}}^{\odot
2}$. Then there exists a constant $C$   such that
\begin{eqnarray*}d_{TV}\left( \frac{F}{\sqrt{E\left[F^{2}\right] }}%
, \mathcal{N}(0,1)\right) \leqslant \frac{C}{E\left[F^{2}\right]}\phi(F),
\end{eqnarray*}
where \begin{align}\phi(F):=\sqrt{\left|\kappa_4\left(I_2(f_{2})\right)\right|}+ \sqrt{\langle
f_{1} \otimes f_{1},f_{2} \underset{1}{{\otimes
}} f_{2}\rangle _{{\mathcal{\mathcal{H}}}^{\otimes 2}}}.\label{def-phi}\end{align}
\end{corollary}
\begin{proof} Let us first recall that
if $g \in \mathcal{H}^{\otimes 2}$, then the  fourth cumulants for $I_2(g)$ satisfies the following (see  (6.6) in \cite{BBNP}):
\begin{align*}
\left|\kappa_4\left(I_2(g)\right)\right| & =16\left(\left\|g \otimes_1 g\right\|_{\mathcal{H}_{\otimes 2}}^2+2\left\|g \widetilde{\otimes_1} g\right\|_{\mathcal{H}^{\otimes 2}}^2\right) 
 \leq 48\left\|g \otimes_1 g\right\|_{\mathcal{H}_{\otimes 2}}^2 . 
\end{align*}
This implies
\begin{align}
16\left\|g \otimes_1 g\right\|_{\mathcal{H}_{\otimes 2}}^2 \leq \left|\kappa_4\left(I_2(g)\right)\right|  
 \leq 48\left\|g \otimes_1 g\right\|_{\mathcal{H}_{\otimes 2}}^2 .\label{4th-cum-2nd-chaos}
\end{align}
For  $d=1$ and $ N=2$, the estimate \eqref{main-ineq.1} yields
\begin{align*}d_{TV}\left( \frac{F}{\sqrt{E\left[F^{2}\right] }}%
, \mathcal{N}(0,1)\right) &\leqslant \frac{C}{E\left[F^{2}\right]}
\left(\Vert
f_{2}\underset{1}{{\otimes }}f_{2}\Vert_{{\mathcal{\mathcal{H}}}^{\otimes {2}}}+ \sqrt{\langle
f_{1} \otimes f_{1},f_{2} \underset{1}{{\otimes
}} f_{2}\rangle _{{\mathcal{\mathcal{H}}}^{\otimes 2}}}\right)\\
&\leqslant \frac{C}{E\left[F^{2}\right]}
\left(\sqrt{\left|\kappa_4\left(I_2(f_{2})\right)\right|}+ \sqrt{\langle
f_{1} \otimes f_{1},f_{2} \underset{1}{{\otimes
}} f_{2}\rangle _{{\mathcal{\mathcal{H}}}^{\otimes 2}}}\right),
\end{align*}
where the last inequality follows from \eqref{4th-cum-2nd-chaos}. Therefore, the desired result is obtained.

\end{proof}

\begin{corollary}\label{corol-stat}
Consider a centered stationary Gaussian sequence $\left\{
Z_{k}, k\geq0\right\}$ with covariance $\rho(k):=\mathbb{E}(Z_{0}Z_{k})$ for all $k\geq0$ such that $\rho(0)>0$,  and $|\rho(k)|\leq1$ for large~$k$. We put
$\rho(k)=\rho(-k)$ for $k<0$. Let  $F_n:=\frac{1}{\sqrt{n}}\sum_{i=0}^{n-1}g_m(Z_i)$ with $g_m(x):=\sum_{k=d}^{m} \lambda_{g_m, 2 k} H_{2 k}\left({x}/{\sqrt{\rho(0)}}\right)$, 
where $d,m \geq   1$, and  for every $k=d, \ldots, m, \lambda_{g_m, 2 k} \in \mathbb{R}$  and  $\lambda_{g_m, 2m}=\rho(0)^{m}$. Then there exists a positive constant $C$ depending only on $m$ such that
\begin{eqnarray}d_{TV}\left( \frac{F_n}{\sqrt{\mathbb{E}[F_n^2]}}%
, \mathcal{N}(0,1)\right) \leqslant \frac{C}{\mathbb{E}[F_n^2]}
\frac{1}{\sqrt{n}}\left[\left(\sum_{|k| <
n}|\rho(k)|^{\frac{4}{3}}\right)^{\frac32}
+\left(\sum_{k=0}^{n-1}\left|\rho(k)\right|^{2d}\right)\left(\sum_{k=0}^{n-1}\left|\rho(k)\right|^2\right)^{\frac12}\right].\label{main-ineq.2}
\end{eqnarray} 
As a consequence, fix an even integer $q \geqslant 2$, and define the power variation
\[
Q_{q, n}(Z):=\frac{1}{n} \sum_{i=0}^{n-1} Z_i^q,
\]
then  there exists a positive constant $C$ depending only on $q$ such that
\begin{eqnarray}&&d_{TV}\left( \frac{Q_{q, n}(Z)-\mathbb{E}Q_{q, n}(Z)}{\sqrt{Var\left(Q_{q, n}(Z)\right)}}%
, \mathcal{N}(0,1)\right)\nonumber\\&& \leqslant \frac{C}{Var\left(Q_{q, n}(Z)\right)}
\frac{1}{\sqrt{n}}\left[\left(\sum_{|k| <
n}|\rho(k)|^{\frac{4}{3}}\right)^{\frac32}
+\left(\sum_{k=0}^{n-1}\left|\rho(k)\right|^2\right)^{\frac32}\right].\label{x2-ineq.case}
\end{eqnarray}
\end{corollary}
\begin{proof}
The first step is to identify the centered stationary Gaussian sequence $\left\{
Z_{k}/\sqrt{\rho(0)}, k\geq0\right\}$  with an  isonormal Gaussian process  $\{ W(h), h \in \mathcal{H}\}$ for some Hilbert space $\mathcal{H}$.
  Suppose that $\mathcal{H}$ is a Hilbert space and let $\left\{\varepsilon_i, i \geq 0\right\}$ be a family of $\mathcal{H}$ such that $\left\langle \varepsilon_i, \varepsilon_j\right\rangle_{\mathcal{H}}=\rho(i-j)$ for each $i, j \geq 0$. In this situation, if $\{W(h): h \in \mathcal{H}\}$ is an isonormal Gaussian process, then the sequence $ \left\{Z_{k}/\sqrt{\rho(0)}, k \geq 0\right\}$ has the same law as $\left\{W\left(\varepsilon_k\right), k \geq 0\right\}$ and we can assume, without any loss of generality, that 
  \begin{equation}Z_k/\sqrt{\rho(0)}=W\left(\varepsilon_k\right)= I_1\left(\varepsilon_k\right),\quad  k \geq 0.\label{eq:W-Z}
  \end{equation}
   Hence  we can write, for every $i \geq  0$,
$$
g_m\left(Z_i\right)=\sum_{k=d}^{m} \lambda_{g_m, 2 k} H_{2 k}\left(\frac{Z_i}{\sqrt{\rho(0)}}\right)=\sum_{k=d}^{m} \lambda_{g_m, 2 k} I_{2 k}\left(\varepsilon_i^{\otimes 2 k}\right),
$$
which implies
\begin{equation}F_n=\sum_{k=d}^{m} I_{2 k}\left(\frac{\lambda_{g_m, 2 k}}{\sqrt{n}}\sum_{i=0}^{n-1}\varepsilon_i^{\otimes 2 k}\right)=:
 \sum_{k=d}^{m} I_{2 k}\left(f_{2k}\right),\label{decomp-F}
\end{equation}
where $f_{2k}=\frac{\lambda_{g_m, 2 k}}{\sqrt{n}}\sum_{i=0}^{n-1}\varepsilon_i^{\otimes 2 k}$.\\
 Furthermore, in view of \eqref{eq:W-Z},  we have, for every  $d\leq p<
q\leq m$,
\begin{eqnarray}
\left|\langle f_{2p} \otimes f_{2p},f_{2q} \underset{2q - 2p}{{\otimes
}} f_{2q}\rangle _{{\mathcal{\mathcal{H}}}%
^{\otimes
4p}}\right|&=&\frac{\lambda_{g_m, 2 p}^2\lambda_{g_m,2q}^2}{n^2}\left|\sum_{k_1,k_2,k_3,k_4=0}^{n-1}\rho(k_1-k_3)^{2p}\rho(k_2-k_4)^{2p}\rho(k_3-k_4)^{2q-2p}\right|\nonumber\\
&=&\frac{\lambda_{g_m, 2 p}^2\lambda_{g_m,2q}^2}{n}\left|\sum_{j_1,j_2,j_3=0}^{n-1}\rho(j_1)^{2p}\rho(j_2)^{2p}\rho(j_3)^{2q-2p}\right|\nonumber\\
&\leq&
\frac{C}{n}\left(\sum_{j=0}^{n-1}\left|\rho(j)\right|^{2d}\right)^2\left(\sum_{j=0}^{n-1}\left|\rho(j)\right|^2\right),\label{estimate1}
\end{eqnarray}
where we used  the change of variables $j_1=k_1-k_3$, $j_2=k_2-k_4$
and $j_3=k_3-k_4$, and $|\rho(k)|\leq1$. \\
According to \cite{NZ}, recall that  if   $M \geq 2 $ is an integer, then 
\begin{equation}
\sum_{\underset{1 \leq j \leq
M}{\left|k_{j}\right| \leq n}}|\rho(\mathbf{k} \cdot \mathbf{v})|
\prod_{j=1}^{M}\left|\rho\left(k_{j}\right)\right| \leq
C\left(\sum_{|k| \leq n}|\rho(k)|^{1+\frac{1}{M}}\right)^{M}, \label{NZ-ineq}
\end{equation}
where $\mathbf{k}=\left(k_{1}, \ldots, k_{M}\right)$ and $\mathbf{v}
\in \mathbb{R}^{M}$ is a fixed vector whose components are 1 or -1.\\
 Next,  suppose $1\leq r\leq 2p-1$ with $d\leq p\leq m$. Making the change of variables $k_{1}-k_{2}=j_{1},
k_{2}-k_{4}=j_{2}$ and $k_{3}-k_{4}=j_{3}$, and   using $|\rho(k)|\leq1$ and 
the  inequality \eqref{NZ-ineq}, we obtain 
\begin{eqnarray}
\Vert
f_{2p}\underset{r}{{\otimes }}f_{2p}\Vert_{{\mathcal{\mathcal{H}}}%
^{\otimes {2p-2r}}}^{2}  &\leqslant &
\frac{C}{n^{2}}\sum_{k_{1},k_{2},k_{3},k_{4}=0}^{n-1}
\rho(k_{1}-k_{2})^{r} \rho(k_{3}-k_{4})^{r}\rho(k_{1}-k_{3})^{2p-r} \rho(k_{2}-k_{4})^{2p-r} \nonumber\\
&\leqslant &\frac{C}{n^{2}}
\sum_{k_{1},k_{2},k_{3},k_{4}=0}^{n-1}\rho(k_{1}-k_{2})\rho(k_{3}-k_{4})\rho(k_{1}-k_{3})\rho(k_{2}-k_{4})
\nonumber\\
&=&\frac{C}{n}  \sum_{\underset{i=1,2,3}{\left|j_{i}\right| <
n}}\left|\rho\left(j_{1}\right) \rho\left(j_{2}\right)
\rho\left(j_{3}\right) \rho\left(j_{1}+j_{2}-j_{3}\right)\right|\nonumber\\
&\leq&\frac{C}{n}\left(\sum_{|k| <
n}|\rho(k)|^{\frac{4}{3}}\right)^{3}.\label{estimate2}
\end{eqnarray}%
Combining Theorem \ref{CLT-finite-sum}, \eqref{decomp-F}, \eqref{estimate1} and \eqref{estimate2} gives the desired estimate \eqref{main-ineq.2}.\\
Now we prove \eqref{x2-ineq.case}. Fix an even integer $q \geqslant 2$. Let $c_{q, 2 k}=\frac{1}{(2 k)!} \int_{-\infty}^{\infty} \frac{e^{-x^2 / 2}}{\sqrt{2 \pi}} x^q H_{2 k}(x) d x$ be the coefficients of the monomial $x^q$ expanded in the basis of Hermite polynomials:
$$
x^q=\sum_{k=0}^{q / 2} c_{q, 2 k} H_{2 k}(x).
$$
It is known  that the coefficients $c_{q, 2 k}$ can be expressed as: 
$$
c_{q, 2 k}=\frac{q!}{2^{\frac{q}{2}-k}\left(\frac{q}{2}-k\right)!(2 k)!}, \quad 0\leq k\leq q/2 .
$$
Note that $c_{q, q}=1$.
 Since \[\left(\frac{Z_i}{\sqrt{\rho(0)}}\right)^q=\sum_{k=0}^{q / 2} c_{q, 2 k} H_{2 k}\left(\frac{Z_i}{\sqrt{\rho(0)}}\right)\]
and for every $1\leq k\leq {q / 2}$, $\mathbb{E} H_{2 k}\left(\frac{Z_i}{\sqrt{\rho(0)}}\right)=0$, we have
$$
  Q_{q, n}(Z)-\mathbb{E}Q_{q, n}(Z)=\frac{\rho(0)^{\frac{q}{2}}}{n} \sum_{i=0}^{n-1} \sum_{k=1}^{q / 2} c_{q, 2 k} H_{2 k}\left(\frac{Z_i}{\sqrt{\rho(0)}}\right),
$$
which satisfies the conditions of Corollary \ref{corol-stat} with $d=1$. Thus, using \eqref{main-ineq.2},  the estimate \eqref{x2-ineq.case} follows.
\end{proof}

\begin{example}
 Let us apply Corollary \ref{corol-stat} to the particular case  
\[
Z_{k}^H:=B_{k+1}^{H}-B_{k}^{H}, \quad k\geq0,
\]
where $B^{H}$ is a fractional Brownian motion  of Hurst parameter $H\in(0,1)$. \\
Observe that $Z^H$ is a centered stationary Gaussian sequence with variance $\rho(0)=1$,
and  covariance given by
\[
\rho(k)=E\left[Z_{0} Z_{k}\right]=\frac{1}{2}\left(|k+1|^{2
H}+|k-1|^{2 H}-2|k|^{2 H}\right), \quad  k\geq0.
\]
Moreover,
\begin{equation}
\rho(k)=H(2 H-1)|k|^{2 H-2}+o\left(|k|^{2 H-2}\right), \quad \text {
as }|k| \rightarrow \infty. \label{cov-fbm}
\end{equation}
Fix an even integer $q\geq2$, and define 
$$
Q_{q, n}(Z^H):=\frac{1}{n} \sum_{i=0}^{n-1}\left(Z_{k}^{H}\right)^q=\frac{1}{n} \sum_{i=0}^{n-1}\left(B_{k+1}^{H}-B_{k}^{H}\right)^q.
$$
In view of  \eqref{x2-ineq.case} and \eqref{cov-fbm},   we can immediately deduce that
\begin{eqnarray}d_{TV}\left( \frac{Q_{q, n}(Z^H)-\mathbb{E}Q_{q, n}(Z^H)}{\sqrt{Var\left(Q_{q, n}(Z^H)\right)}}%
, \mathcal{N}(0,1)\right)\leq  C\times \left\{
\begin{aligned} n^{-\frac12},\quad \mbox{ if } 0<H<\frac85,\\ n^{-\frac12}\log^{3/2}(n),\quad \mbox{if
} H=\frac58,\\ n^{4H-3},\quad \mbox{if } \frac58\leq H<\frac34. \end{aligned}%
\right.
\end{eqnarray}
\end{example}

\section{\textbf{A quantitative bound in the Kolmogorov distance  for a random ratio}}\label{sect:ratio}
 In this section we provide an upper bound  for the
Kolmogorov  distance between  a ratio of multiple stochastic integrals and a
  Gaussian random variable. For these purposes we require some assumptions  on those multiple integrals.\\

\noindent \textbf{Assumption} $\mathbf{\left(\mathcal{A}_1\right)}$:
We assume that $\{G_T,T>0\}$ is a    stochastic process satisfying, for $T>0$,
\begin{eqnarray}
G_T>0 \mbox{ almost surely, and }\left|\frac{\mathbb{E}G_T}{\rho \sqrt{\lambda_T}}-1\right|\rightarrow0 \, \mbox{ and }\,
\mathbb{E}\left[(G_T-\mathbb{E}G_T)^2\right]\longrightarrow\sigma_1^2
\label{estimate-EG_T}
\end{eqnarray}
  as  $\lambda_T\rightarrow\infty$ for some positive constants $\rho>0,\ \sigma_1>0$, whereas $\lambda_T>0$ for all $T>0$, and 
\begin{eqnarray}&G_T-\mathbb{E}G_T=V_T+\frac{S_T}{\sqrt{\lambda_T}},\quad T>0,\label{relation-D-V-R}
\end{eqnarray}
with  \begin{align}V_T=I_2(g_T)\mbox{ and } S_T\in \mathcal{H}_1\oplus\mathcal{H}_2 \mbox{ such that  }  g_T\in {\mathcal{H}}^{\odot
2}\mbox{ and } \frac{\|S_T\|_{L^2(\Omega)}}{\sqrt{\lambda_T}}\rightarrow0,\label{condition-V-R}\end{align}
as $\lambda_T\rightarrow\infty$.\\
\noindent \textbf{Assumption} $\mathbf{\left(\mathcal{A}_2\right)}$: We suppose that      $\{(F_T, U_T,\mu_T),T>0\}$ is
  a stochastic process  that satisfies $F_T~=~I_1(f_T),$ $f_T\in {\mathcal{H}},\, U_T \in{\mathcal{H}}_{1}\oplus{\mathcal{H}}_{2},
\mbox{ and } \mu_T\in \mathbb{R}$ such that, for some positive constant  $\sigma_2>0$, 
\begin{eqnarray} \mathbb{E}\left[F_T^2\right]\longrightarrow\sigma_2^2,
\mbox{ and } \frac{\|U_T\|_{L^2(\Omega)}+|\mu_T|}{\sqrt{\lambda_T}}\rightarrow0,\mbox{ as } \lambda_T\rightarrow\infty.\label{condition-A-a}\end{eqnarray}

\begin{theorem}\label{main-thm} Let $\{G_T,T>0\}$  and $\{(F_T,U_T,\mu_T),T>0\}$
  satisfy the assumptions $\left(\mathcal{A}_1\right)$ and $\mathbf{\left(\mathcal{A}_2\right)}$, respectively. Then, there exists
a constant $C>0$ (independent of $T$) such that, for all $T>0$,
\begin{eqnarray*}
&&d_{Kol}\left(\frac{G_T-\mathbb{E}G_T+F_T+\frac{1}{\sqrt{\lambda_T}}\left(U_T+\mu_T\right)}{\frac{1}{\rho
\sqrt{\lambda_T}}G_{T}},   \mathcal{N}\left(0,\sigma^2\right)   \right)\nonumber\\&\leq&
C\left(\sqrt{\left|\kappa_4\left(V_T\right)\right|}+ \sqrt{\langle
f_T \otimes f_T,g_T \underset{1}{{\otimes
}} g_T\rangle _{{\mathcal{\mathcal{H}}}^{\otimes 2}}}\right)
+C{\lambda_T^{\frac14}}\left|\frac{\mathbb{E}G_T}{\rho
\sqrt{\lambda_T}}-1\right|+ C\left| \mathbb{E}[F_T^2] -
\sigma_2 ^2\right|\nonumber\\&&+ C\left| \mathbb{E}[( G_T-\mathbb{E} G_T)^2] -
\sigma_1 ^2\right|+\frac{C}{\sqrt{\lambda_T}}\left(\|S_T\|_{L^2(\Omega)}+\|U_T\|_{L^2(\Omega)}+|\mu_T|\right), 
\end{eqnarray*}
where the constants  $\rho$ and $\sigma^2:=\sigma_1^2+\sigma_2^2 $ are defined by
\eqref{estimate-EG_T} and \eqref{condition-A-a}, and $\{V_T,T>0\}$ and $\{S_T,T>0\}$ are  the multiple integrals  given
in \eqref{relation-D-V-R}.
\end{theorem}
\begin{proof}
  For $z \in \mathbb{R}, T>0$, we denote
\[\Lambda_T(z):=\frac{ \left(1-\frac{z}{\rho
\sqrt{\lambda_T}}\right)(G_T-\mathbb{E}G_T)+F_T+\frac{1}{\sqrt{\lambda_T}}\left(U_T+\mu_T\right)}{\frac{1}{\rho
\sqrt{\lambda_T}}\mathbb{E}G_T}.
\]
and 
\begin{eqnarray}K_T(z):=
\left(1-\frac{z}{\rho\sqrt{\lambda_T}}\right)(G_T-\mathbb{E}G_T)+F_T+\frac{1}{\sqrt{\lambda_T}}U_T. \label{exp-M}\end{eqnarray}
Hence 
\begin{eqnarray*}\Lambda_T(z)=\frac{K_T(z)+\frac{1}{\sqrt{\lambda_T}} \mu_T}{\frac{1}{\rho \sqrt{\lambda_T}}\mathbb{E}G_T},\quad z \in \mathbb{R},\
T>0.\label{exp-H}\end{eqnarray*}
 For $T>0$, we define
\begin{eqnarray*}Q_T:=\frac{ G_T-\mathbb{E}G_T +F_T+\frac{1}{\sqrt{\lambda_T}}\left(U_T+\mu_T\right)}{\frac{1}{\rho
\sqrt{\lambda_T}}G_{T}}.\end{eqnarray*}
Observe that, for every $z \in \mathbb{R}$,
$T>0$,
\begin{eqnarray}\left\{Q_T \leq z\right\}=\{\Lambda_T(z) \leq
z\}.\label{link-Q-H}\end{eqnarray}
Now suppose that $\lambda_T>\frac{1}{\rho ^4}$ and
$0<z\leq\lambda_T^{\frac14}$. In this way, we have
   \[1-\frac{z}{\rho \sqrt{\lambda_T}}>0,\quad 1-\frac{z}{\rho \sqrt{\lambda_T}}\geq 1-\frac{1}{\rho \lambda_T^{\frac14}}.\]
Combining this together with  \eqref{exp-M} and the orthogonality properties of multiple integrals, we obtain
\begin{eqnarray}\mathbb{E}\left[K_T^2(z)\right]&=& \left(1-\frac{z}{\rho \sqrt{\lambda_T}}\right)^2
\mathbb{E}\left[(G_T-\mathbb{E}G_T)^2\right]+\mathbb{E}\left[F_T^2\right]+\frac{1}{\lambda_T}\mathbb{E}\left[U_T^2\right]
+2\frac{\mathbb{E}\left[U_T F_T\right]}{\sqrt{\lambda_T}}\nonumber\\&&+2\left(1-\frac{z}{\rho \sqrt{\lambda_T}}\right)
\frac{\mathbb{E}\left[U_T(G_T-\mathbb{E}G_T)\right]}{\sqrt{\lambda_T}}
+2\left(1-\frac{z}{\rho \sqrt{\lambda_T}}\right)
\frac{\mathbb{E}\left[F_TR_T\right]}{\sqrt{\lambda_T}}
\label{decomp-K}\\
&\geq& \left(1-\frac{1}{\rho \lambda_T^{\frac14}}\right)^2
\mathbb{E}\left[(G_T-\mathbb{E}G_T)^2\right]+\mathbb{E}\left[F_T^2\right]+\frac{1}{\lambda_T}\mathbb{E}\left[U_T^2\right]
+2\frac{\mathbb{E}\left[U_T F_T\right]}{\sqrt{\lambda_T}}\nonumber\\&&-2 \left|1-\frac{1}{\rho
\lambda_T^{\frac14}}\right|
\frac{\left|\mathbb{E}\left[U_T(G_T-\mathbb{E}G_T)\right]\right|}{\sqrt{\lambda_T}}
-2 \left|1-\frac{1}{\rho
\lambda_T^{\frac14}}\right|
\frac{\left|\mathbb{E}\left[F_TR_T\right]\right|}{\sqrt{\lambda_T}}\nonumber\\
&=:&\eta_T.\label{estimate-K-1}
\end{eqnarray}
Furthermore, 
\begin{eqnarray}\mathbb{E}\left[K_T^2(z)\right]&\leq&  
\mathbb{E}\left[(G_T-\mathbb{E}G_T)^2\right]+\mathbb{E}\left[F_T^2\right]+\frac{1}{\lambda_T}\mathbb{E}\left[U_T^2\right]
+2\frac{\mathbb{E}\left[U_T F_T\right]}{\sqrt{\lambda_T}}\nonumber\\&&+2  
\frac{\left|\mathbb{E}\left[U_T(G_T-\mathbb{E}G_T)\right]\right|}{\sqrt{\lambda_T}}
+2 
\frac{\left|\mathbb{E}\left[F_TR_T\right]\right|}{\sqrt{\lambda_T}}
\nonumber\\
&=:&\delta_T. \label{estimate-K-2}
\end{eqnarray}
So, by \eqref{estimate-K-1}  and \eqref{estimate-K-2}, we get, for every  $\lambda_T>\frac{1}{\rho ^4}$,
\[\eta_T\leq \sup_{0<z\leq\lambda_T^{\frac14}}\mathbb{E}\left[K_T^2(z)\right]\leq \delta_T.\]
According to \eqref{estimate-EG_T} and
\eqref{condition-A-a}, $\eta_T\longrightarrow  \sigma^2$ and $\delta_T\longrightarrow  \sigma^2$ as  
$\lambda_T\rightarrow\infty$, so, there exist $c_1,c_2>0, \lambda_{T_0}>\frac{1}{\rho ^4}$ such that 
\begin{eqnarray}c_1\leq\sup_{\lambda_{T}>\lambda_{T_0},\ 0<z\leq\lambda_T^{\frac14}}\mathbb{E}\left[K_T^2(z)\right]\leq c_2. \label{sup of M}\end{eqnarray}
Next, in view of \eqref{link-Q-H}, we obtain
\begin{eqnarray}
 &&\left|\mathbb{P}\left\{Q_T \leq z\right\}-\mathbb{P}\{ \mathcal{N}(0,\sigma^2) 
\leq z\}\right|=\left|\mathbb{P}\left\{\Lambda_T(z) \leq
z\right\}-\mathbb{P}\{ \mathcal{N}(0,\sigma^2)  \leq
z\}\right|\nonumber\\&\leq&\left|\mathbb{P}\left\{\frac{K_T(z)}{\frac{1}{\sigma}\sqrt{\mathbb{E}\left[K_T^2(z)\right]}}
\leq \frac{z\frac{1}{\rho
\sqrt{\lambda_T}}\mathbb{E}G_T-\frac{1}{\sqrt{\lambda_T}}\mu_T}{\frac{1}{\sigma}\sqrt{\mathbb{E}\left[K_T^2(z)\right]}}\right\}
-\mathbb{P}\left\{ \mathcal{N}(0,\sigma^2)  \leq \frac{z\frac{1}{\rho
\sqrt{\lambda_T}}\mathbb{E}G_T-\frac{1}{\sqrt{\lambda_T}}\mu_T}{\frac{1}{\sigma}\sqrt{\mathbb{E}\left[K_T^2(z)\right]}}\right\}\right|\nonumber\\&&+
\left|\mathbb{P}\left\{ \mathcal{N}(0,\sigma^2) \leq \frac{z\frac{1}{\rho
\sqrt{\lambda_T}}\mathbb{E}G_T-\frac{1}{\sqrt{\lambda_T}}\mu_T}{\frac{1}{\sigma}\sqrt{\mathbb{E}\left[K_T^2(z)\right]}}\right\}
-\mathbb{P}\left\{ \mathcal{N}(0,\sigma^2)  \leq \frac{z\frac{1}{\rho
\sqrt{\lambda_T}}\mathbb{E}G_T}{\frac{1}{\sigma}\sqrt{\mathbb{E}\left[K_T^2(z)\right]}}\right\}\right|\nonumber\\&&+
\left|\mathbb{P}\left\{ \mathcal{N}(0,\sigma^2) \leq \frac{z\frac{1}{\rho
\sqrt{\lambda_T}}\mathbb{E}G_T}{\frac{1}{\sigma}\sqrt{\mathbb{E}\left[K_T^2(z)\right]}}\right\}
-\mathbb{P}\{ \mathcal{N}(0,\sigma^2)  \leq z\}\right|\nonumber\\
&=:& A(z,T)+B(z,T)+D(z,T).\label{triang.-kol}
\end{eqnarray}
Observe that 
\begin{align}A(z,T)\leq d_{{Kol}}\left(\frac{K_T(z)}{\frac{1}{\sigma}\sqrt{\mathbb{E}\left[K_T^2(z)\right]}}, \mathcal{N}(0,\sigma^2) \right).\label{estimate-A(z,T)}
\end{align}
By straightforward calculations, we have, for all $u,v\in \mathbb{R}$,
\begin{eqnarray}\left|\mathbb{P}\left\{ \mathcal{N}(0,\sigma^2)  \leq
u\right\}-\mathbb{P}\left\{ \mathcal{N}(0,\sigma^2)  \leq
v\right\}\right|\leq\frac{|u-v|}{\sqrt{2\pi}\sigma}e^{-\frac{\min(u^2,v^2)}{2\sigma^2}}, \label{diff-proba-estimate}\end{eqnarray}
and
\begin{eqnarray}\sup_{\{x>0\}}|x|e^{-x^2}<\infty,\qquad
 \sup_{\{x>0\}}x^2e^{- x^2}<\infty.\label{sup-z-exp}
\end{eqnarray}
Further, in view of  \eqref{condition-A-a} and \eqref{sup
of M},
\begin{eqnarray}
  \frac{1}{\sqrt{2\pi}\sigma}\left|\frac{\frac{1}{\sqrt{\lambda_T}}\mu_T}{\sqrt{\mathbb{E}\left[K_T^2(z)\right]}}\right|
\leq C\frac{|\mu_T|}{\sqrt{\lambda_T}}.\label{estimate-mu}
\end{eqnarray}
Hence, it follows from \eqref{diff-proba-estimate}, \eqref{sup-z-exp} and \eqref{estimate-mu} that
\begin{align}B(z,T)\leq  C\frac{|\mu_T|}{\sqrt{\lambda_T}}. \label{estimate-B(z,T)}
\end{align}
Moreover, by using \eqref{estimate-EG_T}, \eqref{condition-A-a}, \eqref{decomp-K}, \eqref{sup of M}, \eqref{diff-proba-estimate}, \eqref{sup-z-exp},  and  the orthogonality properties of multiple integrals,
\begin{align}D(z,T)&\leq\frac{|z|}{\sqrt{2\pi}\sigma}\left|\frac{\frac{\mathbb{E}G_T}{\rho
\sqrt{\lambda_T}}}{\frac{1}{\sigma}\sqrt{\mathbb{E}\left[K_T^2(z)\right]}}-1\right|
e^{-C z^2 \min\left(1,\frac{\frac{\mathbb{E}G_T}{\rho
\sqrt{\lambda_T}}}{\frac{1}{\sigma}\sqrt{\mathbb{E}\left[K_T^2(z)\right]}}\right)^2}\nonumber\\&\leq
\frac{|z|}{\sqrt{2\pi}\sigma}\left|\frac{\frac{\mathbb{E}G_T}{\rho
\sqrt{\lambda_T}}-1}{\frac{1}{\sigma}\sqrt{\mathbb{E}\left[K_T^2(z)\right]}}\right|
e^{-C z^2 \min\left(1,\frac{\frac{\mathbb{E}G_T}{\rho
\sqrt{\lambda_T}}}{\sqrt{\mathbb{E}\left[K_T^2(z)\right]}}\right)^2}
\nonumber\\&+
\frac{|z|}{\sqrt{2\pi}\sigma}\left|\frac{1}{\frac{1}{\sigma}\sqrt{\mathbb{E}\left[K_T^2(z)\right]}}-1\right|
e^{-C \left(\frac{z}{\sqrt{\mathbb{E}\left[K_T^2(z)\right]}}\right)^2
\min\left(\sqrt{\eta_T}, \frac{1}{\rho \sqrt{\lambda_T}}\mathbb{E}G_T
\right)^2}
\nonumber\\&\leq C  |z|\left|\frac{\mathbb{E}G_T}{\rho
\sqrt{\lambda_T}}-1\right| + C\frac{|z|}{\sqrt{\mathbb{E}\left[K_T^2(z)\right]}}\left| \mathbb{E}\left[K_T^2(z)\right] -\sigma^2\right|
e^{-C \left(\frac{z}{\sqrt{\mathbb{E}\left[K_T^2(z)\right]}}\right)^2}
\nonumber\\&\leq C\left[{\lambda_T^{\frac14}}\left|\frac{\mathbb{E}G_T}{\rho
\sqrt{\lambda_T}}-1\right| +  \left| \mathbb{E}[( G_T-\mathbb{E} G_T)^2] -
\sigma_1 ^2\right|+\left| \mathbb{E}[F_T^2] -
\sigma_2 ^2\right|+\frac{\|U_T\|_{L^2(\Omega)}+|\mu_T|}{\sqrt{\lambda_T}}\right].\label{estimate-D(z,T)}\end{align}
Consequently, in view of  \eqref{triang.-kol}, \eqref{estimate-A(z,T)}, \eqref{estimate-B(z,T)} and \eqref{estimate-D(z,T)}, we obtain, for
every $ \lambda_{T}~>~\lambda_{T_0},$  $0~<~z~\leq~\lambda_T^{\frac14}$,
\begin{eqnarray}
 &&\left|\mathbb{P}\left\{Q_T \leq
 z\right\}-\mathbb{P}\{ \mathcal{N}(0,\sigma^2)  \leq z\}\right| \nonumber\\
&\leq&
d_{{Kol}}\left(\frac{K_T(z)}{\frac{1}{\sigma}\sqrt{\mathbb{E}\left[K_T^2(z)\right]}}
, \mathcal{N}(0,\sigma^2) \right)+C{\lambda_T^{\frac14}}\left|\frac{1}{\rho
\sqrt{\lambda_T}}\mathbb{E}G_T-1\right|+ C\left| \mathbb{E}\left[( G_T-\mathbb{E} G_T)^2\right] -
\sigma_1^2\right|\nonumber\\&&+ C\left| \mathbb{E}\left[F_T^2\right] -
\sigma_2^2\right|+\frac{C}{\sqrt{\lambda_T}}\left(\|U_T\|_{L^2(\Omega)}+|\mu_T|\right).\label{{Kol}-F-z<}
\end{eqnarray}
 On the other
hand, using  \eqref{estimate-EG_T}--\eqref{condition-A-a}, \eqref{exp-M} and  the orthogonality properties of multiple integrals, we
obtain, for every $\lambda_T>\lambda_{T_0}$, $0<z\leq \lambda_T^{\frac14}$,
\begin{eqnarray}
\phi(K_T(z))
&\leq& \phi\left(\left(1-\frac{z}{\rho\sqrt{\lambda_T}}\right)(G_T-\mathbb{E}G_T)+F_T\right) 
+\frac{C}{\sqrt{\lambda_T}} \|U_T\|_{L^2(\Omega)} \nonumber\\
&\leq&\phi\left(G_T-\mathbb{E}G_T+F_T\right) 
+\frac{C}{\sqrt{\lambda_T}}\|U_T\|_{L^2(\Omega)}\nonumber\\
&\leq&\phi\left( V_T+F_T\right) +\frac{C}{\sqrt{\lambda_T}}\left(\|S_T\|_{L^2(\Omega)}+\|U_T\|_{L^2(\Omega)}\right),\label{4cumulant-F-A}
\end{eqnarray}
where $\phi(.)$ is defined by \eqref{def-phi}.
Applying  Corollary \ref{1st-2nd-chaos} together  with \eqref{dKol:dTV}, \eqref{sup of M} and \eqref{4cumulant-F-A}, we have, for
every $\lambda_T>\lambda_{T_0}$, $0<z\leq \lambda_T^{\frac14}$, 
\begin{align}d_{{Kol}}\left(\frac{K_T(z)}{\frac{1}{\sigma}\sqrt{\mathbb{E}\left[K_T^2(z)\right]}}
, \mathcal{N}(0,\sigma^2) \right)\leq C\left[\phi\left( V_T+F_T\right) +\frac{\|S_T\|_{L^2(\Omega)}+\|U_T\|_{L^2(\Omega)}}{\sqrt{\lambda_T}}\right].\label{upper-M-Kol}
\end{align}
Thus, in view of \eqref{{Kol}-F-z<} and \eqref{upper-M-Kol}, we deduce that,
 for
every   $\lambda_{T}>\lambda_{T_0}$,
\begin{align}
 &\sup_{\left\{0\leq
z\leq\lambda_T^{\frac14}\right\}}\left|\mathbb{P}\left\{Q_T \leq
 z\right\}-\mathbb{P}\{ \mathcal{N}(0,\sigma^2)  \leq z\}\right| \nonumber\\
&\leq
C \left[\phi\left( V_T+F_T\right)  +{\lambda_T^{\frac14}}\left|\frac{1}{\rho
\sqrt{\lambda_T}}\mathbb{E}G_T-1\right|+ \left| \mathbb{E}\left[( G_T-\mathbb{E} G_T)^2\right] -
\sigma_1^2\right|+ \left| \mathbb{E}\left[F_T^2\right] -
\sigma_2^2\right|\right.\nonumber\\&\quad\left.+\frac{\|S_T\|_{L^2(\Omega)}+\|U_T\|_{L^2(\Omega)}+|\mu_T|}{\sqrt{\lambda_T}}\right].\label{Kol-Q1}
\end{align}
Furthermore, in view of \eqref{Kol-Q1}  and the
fact that,   for every $ x>0$,   \[\mathbb{P}\left\{\mathcal{N}(0,\sigma^2) \geq x\right\}=\int_{x}^{\infty}\frac{1}{\sqrt{2\pi}\sigma} e^{-\frac{y^2}{2\sigma^2}}dy
\leq e^{-\frac{x^2}{4\sigma^2}}\int_{x}^{\infty}\frac{1}{\sqrt{4\pi}\sigma} e^{-\frac{y^2}{2\sigma^2}}dy \leq C
e^{-\frac{x^2}{4\sigma^2}},\] we  obtain
\begin{eqnarray}
&&\sup_{\left\{z>\lambda_T^{\frac14}\right\}}\left|\mathbb{P}\left\{Q_T \leq z\right\}-\mathbb{P}\{ \mathcal{N}(0,\sigma^2) 
\leq z\}\right|=\sup_{\left\{z>\lambda_T^{\frac14}\right\}}\left|\mathbb{P}\left\{Q_T > z\right\}-\mathbb{P}\{ N(0,\sigma^2) 
> z\}\right|\nonumber\\ &\leq& \mathbb{P}\left\{Q_T
  > \lambda_T^{\frac14}\right\}+\mathbb{P}\left\{ \mathcal{N}(0,\sigma^2)  > \lambda_T^{\frac14}\right\}\nonumber \\
&\leq& \left|\mathbb{P}\left\{Q_T \leq
\lambda_T^{\frac14}\right\}-\mathbb{P}\left\{ \mathcal{N}(0,\sigma^2)  \leq
\lambda_T^{\frac14}\right\}\right|
+2 \mathbb{P}\left\{ \mathcal{N}(0,\sigma^2)  \geq \lambda_T^{\frac14}\right\}\nonumber \\
 &\leq& \sup_{\left\{0\leq
z\leq\lambda_T^{\frac14}\right\}}\left|\mathbb{P}\left\{Q_T \leq
 z\right\}-\mathbb{P}\{ \mathcal{N}(0,\sigma^2)  \leq z\}\right|+Ce^{-\frac{\sqrt{\lambda_T}}{4\sigma^2}}\nonumber\\
&\leq&
C \left[\phi\left( V_T+F_T\right)  +{\lambda_T^{\frac14}}\left|\frac{1}{\rho
\sqrt{\lambda_T}}\mathbb{E}G_T-1\right|+ \left| \mathbb{E}\left[( G_T-\mathbb{E} G_T)^2\right] -
\sigma_1^2\right|+ \left| \mathbb{E}\left[F_T^2\right] -
\sigma_2^2\right|\right.\nonumber\\&&\quad\left.+\frac{\|S_T\|_{L^2(\Omega)}+\|U_T\|_{L^2(\Omega)}+|a_T|}{\sqrt{\lambda_T}}\right].\label{Kol-Q2}
\end{eqnarray}
Now let us study the case when $z<0$. Suppose   $-\lambda_T^{\frac14}\leq z<0$, then, following similar arguments as in the above case when $0<z\leq\lambda_T^{\frac14}$, there exists $\lambda_{T_0}$ such that, for
every   $\lambda_{T}>\lambda_{T_0}$,
\begin{align}
 &\sup_{\left\{-\lambda_T^{\frac14}\leq z<0\right\}}\left|\mathbb{P}\left\{Q_T \leq
 z\right\}-\mathbb{P}\{ \mathcal{N}(0,\sigma^2)  \leq z\}\right| \nonumber\\
&\leq
C \left[\phi\left( V_T+F_T\right)  +{\lambda_T^{\frac14}}\left|\frac{1}{\rho
\sqrt{\lambda_T}}\mathbb{E}G_T-1\right|+ \left| \mathbb{E}\left[( G_T-\mathbb{E} G_T)^2\right] -
\sigma_1^2\right|+ \left| \mathbb{E}\left[F_T^2\right] -
\sigma_2^2\right|\right.\nonumber\\&\quad\left.+\frac{\|S_T\|_{L^2(\Omega)}+\|U_T\|_{L^2(\Omega)}+|\mu_T|}{\sqrt{\lambda_T}}\right],\label{Kol-Q1-z-neg}
\end{align}
where we used $\left(1-\frac{z}{\rho \sqrt{\lambda_T}}\right)\geq1$ and $\left(1-\frac{z}{\rho \sqrt{\lambda_T}}\right)\leq 2$ for $\lambda_T>\frac{1}{\rho ^4}$.\\
Further, combining \eqref{Kol-Q1-z-neg}  and the
fact that  \[\mathbb{P}\left\{\mathcal{N}(0,\sigma^2) \leq -x\right\}=\int_{-\infty}^{-x}\frac{1}{\sqrt{2\pi}\sigma} e^{-\frac{y^2}{2\sigma^2}}dy=\int_{x}^{\infty}\frac{1}{\sqrt{2\pi}\sigma} e^{-\frac{y^2}{2\sigma^2}}dy
 \leq C
e^{-\frac{x^2}{4\sigma^2}} \mbox{ for all } x>0 ,\] we deduce 
\begin{eqnarray}
&&\sup_{\left\{z<-\lambda_T^{\frac14}\right\}}\left|\mathbb{P}\left\{Q_T \leq z\right\}-\mathbb{P}\{ \mathcal{N}(0,\sigma^2) 
\leq z\}\right|\nonumber\\ &\leq& \mathbb{P}\left\{Q_T
  \leq - \lambda_T^{\frac14}\right\}+\mathbb{P}\left\{ \mathcal{N}(0,\sigma^2)  \leq - \lambda_T^{\frac14}\right\}\nonumber \\
&\leq& \left|\mathbb{P}\left\{Q_T \leq
-\lambda_T^{\frac14}\right\}-\mathbb{P}\left\{ \mathcal{N}(0,\sigma^2)  \leq
-\lambda_T^{\frac14}\right\}\right|
+2 \mathbb{P}\left\{ \mathcal{N}(0,\sigma^2)  \leq -\lambda_T^{\frac14}\right\}\nonumber \\
 &\leq& \sup_{\left\{-\lambda_T^{\frac14}\leq z<0\right\}}\left|\mathbb{P}\left\{Q_T \leq
 z\right\}-\mathbb{P}\{ \mathcal{N}(0,\sigma^2)  \leq z\}\right|+Ce^{-\frac{\sqrt{\lambda_T}}{4\sigma^2}}\nonumber\\
&\leq&
C \left[\phi\left( V_T+F_T\right)  +{\lambda_T^{\frac14}}\left|\frac{1}{\rho
\sqrt{\lambda_T}}\mathbb{E}G_T-1\right|+ \left| \mathbb{E}\left[( G_T-\mathbb{E} G_T)^2\right] -
\sigma_1^2\right|+ \left| \mathbb{E}\left[F_T^2\right] -
\sigma_2^2\right|\right.\nonumber\\&&\quad\left.+\frac{\|S_T\|_{L^2(\Omega)}+\|U_T\|_{L^2(\Omega)}+|\mu_T|}{\sqrt{\lambda_T}}\right].\label{Kol-Q2-z-neg}
\end{eqnarray}
 Therefore, in view of \eqref{Kol-Q1}, \eqref{Kol-Q2}, \eqref{Kol-Q1-z-neg} and \eqref{Kol-Q2-z-neg}, we conclude that
\begin{eqnarray*}
&&\sup_{\{z\in\mathbb{R}\}}\left|\mathbb{P}\left\{Q_T \leq
z\right\}-\mathbb{P}\{ \mathcal{N}(0,\sigma^2)  \leq z\}\right|\nonumber\\ 
&&\leq
C \left[\phi\left( V_T+F_T\right)  +{\lambda_T^{\frac14}}\left|\frac{1}{\rho
\sqrt{\lambda_T}}\mathbb{E}G_T-1\right|+ \left| \mathbb{E}\left[( G_T-\mathbb{E} G_T)^2\right] -
\sigma_1^2\right|+ \left| \mathbb{E}\left[F_T^2\right] -
\sigma_2^2\right|\right.\nonumber\\&&\quad\left.+\frac{\|S_T\|_{L^2(\Omega)}+\|U_T\|_{L^2(\Omega)}+|\mu_T|}{\sqrt{\lambda_T}}\right],
\end{eqnarray*}
which completes the proof of Theorem \ref{main-thm}. 
\end{proof}

\end{document}